\documentclass[11pt,letterpaper]{amsart}

\usepackage{amsmath}
\usepackage{amsthm}
\usepackage{amssymb}

\newtheorem{theorem}{Theorem} 
\newtheorem{lemma}[theorem]{Lemma}

\newcommand{\e}{\epsilon}
\newcommand{\F}{\mathbb{F}}
\newcommand{\N}{\mathbb{N}}
\newcommand{\R}{\mathbb{R}}

\newcommand{\ip}[1]{\langle#1\rangle}
\newcommand{\abs}[1]{\lvert#1\rvert}
\newcommand{\bigabs}[1]{\big\lvert#1\big\rvert}

\newcommand{\Biggabs}[1]{\Bigg\lvert#1\Bigg\rvert}
\newcommand{\sums}[1]{\sum_{\substack{#1}}}

\DeclareMathOperator{\E}{E}
\DeclareMathOperator{\RM}{RM}
\DeclareMathOperator{\wt}{wt}
\DeclareMathOperator{\B}{\mathfrak{B}}

\begin{document}

\title[Nonlinearity measures of random Boolean functions]{Nonlinearity measures of random\\Boolean functions}

\author{Kai-Uwe Schmidt}

\date{14 August 2013}

 


\address{Faculty of Mathematics, Otto-von-Guericke University, Universit\"atsplatz~2, 39106 Magdeburg, Germany.}

\email{{\tt kaiuwe.schmidt@ovgu.de}}

\begin{abstract}
The $r$-th order nonlinearity of a Boolean function is the minimum number of elements that have to be changed in its truth table to arrive at a Boolean function of degree at most $r$. It is shown that the (suitably normalised) $r$-th order nonlinearity of a random Boolean function converges strongly for all $r\ge 1$. This extends results by Rodier for $r=1$ and by Dib for $r=2$. The methods in the present paper are mostly of elementary combinatorial nature and also lead to simpler proofs in the cases that $r=1$ or $2$.
\end{abstract}

\maketitle


\section{Introduction and Results}

Let $\F_2$ be a field with two elements. A \emph{Boolean function} $f$ is a mapping from $\F_2^n$ to $\F_2$ and its \emph{truth table} is the list of values $f(x)$ as $x$ ranges over $\F_2^n$ in some fixed order. Let $\B_n$ be the space of Boolean functions on $\F_2^n$. Every $f\in\B_n$ can be written uniquely in the form
\[
f(x_1,\dots,x_n)=\sum_{k_1,\dots,k_n\in\{0,1\}}a_{k_1,\dots,k_n}\,x_1^{k_1}\cdots x_n^{k_n},
\]
where $a_{k_1,\dots,k_n}\in\F_2$. The \emph{degree} of~$f$ is defined to be the algebraic degree of this polynomial.
\par
The \emph{$r$-th order nonlinearity} $N_r(f)$ of a Boolean function $f$ is the minimum number of elements that have to be changed in its truth table to arrive at the truth table of a Boolean function of degree at most $r$. We state this definition more formally as follows. Let $\RM(r,n)$ be the set of Boolean functions in $\B_n$ of degree at most $r$ (which is known as the \emph{Reed-Muller code} of length~$2^n$ and order $r$; see \cite[Chapters~13--15]{MacSlo1977}, for example) and define the \emph{Hamming distance} between $f,g\in\B_n$ to be 
\[
d(f,g)=\bigabs{\{x\in\F_2^n:f(x)\ne g(x)\}}.
\]
Then the $r$-th order nonlinearity of $f$ is
\[
N_r(f)=\min_{g\in\RM(r,n)} d(f,g).
\]
The nonlinearity of Boolean functions is of significant relevance in cryptography since it measures the resistance of a Boolean function against low-degree approximation attacks (see~\cite{KnuRob1996}, for example, and~\cite{Car2010} for more background on the role of Boolean functions in cryptography and error-correcting codes).
\par
Our interest is the distribution of the nonlinearity of Boolean functions. To this end, let $\Omega$ be the set of infinite sequences of elements from $\F_2$ and let $\B$ be the space of functions from $\Omega$ to $\F_2$. For $f\in\B$, we denote the restriction of $f$ to its first $n$ coordinates by $f_n$, which is in $\B_n$. We endow $\B$ with a probability measure defined by
\begin{equation}
\Pr\big[f\in\B:f_n=g\big]=2^{{-2^n}}\quad\text{for all $g\in\B_n$ and all $n\in\N$}.   \label{eqn:def_prob_measure}
\end{equation}
A basic probabilistic method can be used to show that, if $f$ is drawn from~$\B$, equipped with the probability measure defined by~\eqref{eqn:def_prob_measure}, then 
\begin{equation}
\limsup_{n\to\infty}\frac{2^{n-1}-N_r(f_n)}{\sqrt{2^{n-1}\tbinom{n}{r}\log 2}}\le1\quad\text{almost surely}.   \label{eqn:upper_bound}
\end{equation}
This was proved with a weaker convergence mode by Carlet~\cite[Theorem~1]{Car2006}. The aim of this paper is to prove strong convergence of the normalised $r$-th order nonlinearity, which shows that the bound~\eqref{eqn:upper_bound} is best possible.
\begin{theorem}
\label{thm:main}
Let $f$ be drawn at random from $\B$, equipped with the probability measure defined by~\eqref{eqn:def_prob_measure}. Then for all $r\ge1$, as $n\to\infty$,
\begin{equation}
\frac{2^{n-1}-N_r(f_n)}{\sqrt{2^{n-1}\binom{n}{r}\log 2}}\to1\quad\text{almost surely}   \label{eqn:main_Pr}
\end{equation}
and
\begin{equation}
\frac{2^{n-1}-\E[N_r(f_n)]}{\sqrt{2^{n-1}\binom{n}{r}\log 2}}\to1.   \label{eqn:main_E}
\end{equation}
\end{theorem}
\par
Using Fourier analytic methods due to Hal\'asz~\cite{Hal1973}, Rodier~\cite{Rod2006} proved~\eqref{eqn:main_Pr} for $r=1$. More precise estimates on the rate of convergence in this case were given by Litsyn and Shpunt~\cite{LitShp2008}, using different methods. Dib~\cite{Dib2010} used a more combinatorial approach to prove~\eqref{eqn:main_Pr} with a weaker convergence mode for $r=2$. The methods in this paper are mostly of elementary combinatorial nature and also lead to simpler proofs of~\eqref{eqn:main_Pr} in the cases that $r=1$ or $2$.
\par
With the notation as in Theorem~\ref{thm:main}, write $Y_{n,g}=2^n-2d(f_n,g)$ for $g\in\B_n$. In Section~\ref{sec:RM}, we show that most pairs of functions in $\RM(r,n)$~have Hamming distance close to~$2^{n-1}$. Combining this with some large deviation estimates in Section~\ref{sec:LD} then shows that the events
\[
Y_{n,g}\ge\sqrt{2^{n+1}\tbinom{n}{r}\log 2}
\]
are pairwise nearly independent for all $g$ from a large subset of $\RM(r,n)$. This will be the key ingredient for the proof of Theorem~\ref{thm:main}, which will be completed in Section~\ref{sec:proof}.


\section{Some results on Reed-Muller codes}
\label{sec:RM}

In this section, we show that most pairs of functions in $\RM(r,n)$ have Hamming distance close to $2^{n-1}$.
\par
The \emph{weight} of a Boolean function $f$, denoted by $\wt(f)$, is defined to be its Hamming distance to the zero function. For real $x$, write 
\[
A_{r,n}(x)=\bigabs{\{g\in\RM(r,n):\wt(g)\le2^nx\}}.
\]
Our starting point is the following asymptotic characterisation of $A_{r,n}(x)$, which is a special case of a result due to Kaufman, Lovett, and Porat~\cite{Kaufman2012}. 
\begin{lemma}[{\cite[Theorem~3.1]{Kaufman2012}}]
\label{lem:wd_RM}
For all $r\ge 1$, there exists a constant $K_r$ such that
\[
A_{r,n}\bigg(\frac{1-\delta}{2}\bigg)\le \bigg(\frac{1}{\delta}\bigg)^{K_rn^{r-1}}
\]
for all real $\delta$ satisfying $0<\delta\le 1/2$.
\end{lemma}
\par
It should be noted that the case $r=1$ is not covered in~\cite[Theorem~3.1]{Kaufman2012}. Lemma~\ref{lem:wd_RM} however holds trivially in this case, since all but two functions in $\RM(1,n)$ have weight $2^{n-1}$.
\par
We now apply Lemma~\ref{lem:wd_RM} to prove the main result of this section.
\begin{lemma}
\label{lem:frac_RM}
Let $\alpha>0$ be real and let $r\ge 1$ be integral. Then, for all sufficiently large $n$, there exists a subset $S\subset\RM(r,n)$ of cardinality at least $2^{(1-\alpha)\binom{n}{r}}$ such that
\begin{equation}
\bigabs{d(g,h)-2^{n-1}}\le 2^{n-1}/\tbinom{n}{r}\quad\text{for all $g,h\in S$ with $g\ne h$.}   \label{eqn:subset}
\end{equation}
\end{lemma}
\begin{proof}
Let $B_{r,n}$ be the number of functions $g$ in $\RM(r,n)$ satisfying
\[
\bigabs{\wt(g)-2^{n-1}}\ge 2^{n-1}/\tbinom{n}{r}.
\]
Since $\RM(r,n)$ contains the nonzero constant function, there is a bijection between the functions in $\RM(r,n)$ of weight $w$ and the functions in $\RM(r,n)$ of weight $2^n-w$. Therefore,
\[
B_{r,n}=2A_{r,n}\bigg(\frac{1-1/\binom{n}{r}}{2}\bigg)
\]
and so by Lemma~\ref{lem:wd_RM},
\[
\log_2\bigg(\frac{B_{r,n}}{2}\bigg)\le K_rn^{r-1}\log_2\binom{n}{r}\le K_r\binom{n}{r}\frac{r^r}{n}\log_2\binom{n}{r},
\]
where $K_r$ is the same constant as in Lemma~\ref{lem:wd_RM}. Therefore,
\begin{equation}
B_{r,n}\le 2^{\alpha\binom{n}{r}}   \label{eqn:B_bound}
\end{equation}
for all sufficiently large $n$.
\par
Next we construct the set $S$ iteratively as follows. We take $n$ large enough, so that the bound~\eqref{eqn:B_bound} for $B_{r,n}$ holds. Choose a $g\in\RM(r,n)$ to be in $S$ and delete all $u\in\RM(r,n)$ satisfying
\[
\bigabs{d(g,u)-2^{n-1}}\ge 2^{n-1}/\tbinom{n}{r}.
\]
From~\eqref{eqn:B_bound} it is readily verified that the number of deleted functions is at most $2^{\alpha\,\binom{n}{r}}$. We can continue in this way to choose functions of $\RM(r,n)$ to be in $S$, while maintaining the property~\eqref{eqn:subset}, as long as the number of chosen functions times $1+2^{\alpha\,\binom{n}{r}}$ is less than the cardinality of $\RM(r,n)$, namely $2^{1+\binom{n}{1}+\cdots+\binom{n}{r}}$. We can therefore obtain a set $S$ satisfying~\eqref{eqn:subset} and
\[
\abs{S}\ge \frac{2^{1+\binom{n}{1}+\cdots+\binom{n}{r}}}{1+2^{\alpha\,\binom{n}{r}}}\ge \frac{2^{\binom{n}{r}}}{2^{\alpha\,\binom{n}{r}}}
\]
for all sufficiently large $n$.
\end{proof}


\section{Some large deviation estimates}
\label{sec:LD}

In this section, we give some estimates for tail probabilities of sums of independent identically distributed random variables. For $\mathbf{a},\mathbf{b}\in\R^m$, we denote their scalar product by $\ip{\mathbf{a},\mathbf{b}}$.
\begin{lemma}
\label{lem:subnormal}
Let $\mathbf{g}$ and $\mathbf{h}$ be elements of $\{-1,1\}^N$ and let $X$ be drawn at random from $\{-1,1\}^N$, equipped with the uniform probability measure. Write $Y_g=\ip{X,\mathbf{g}}$ and $Y_h=\ip{X,\mathbf{h}}$. Then, for all $t_1,t_2\in\R$,
\[
\E\big[\exp(t_1Y_g+t_2Y_h)\big]\le \exp\big(\tfrac{1}{2}N\big(t_1^2+t_2^2\big)+t_1t_2\ip{\mathbf{g},\mathbf{h}}\big).
\]
\end{lemma}
\begin{proof}
Write $X=(X_1,\dots,X_N)$, $\mathbf{g}=(g_1,\dots,g_N)$, and $\mathbf{h}=(h_1,\dots,h_N)$. Then
\begin{align*}
\E\big[\exp(t_1Y_g+t_2Y_h)\big]
&=\E\Bigg[\prod_{j=1}^N\exp\big(X_j(t_1g_j+t_2h_j)\big)\Bigg]\\
&=\prod_{j=1}^N\E\big[\exp\big(X_j(t_1g_j+t_2h_j)\big)\big]
\end{align*}
using that the $X_j$'s are independent. Since the $X_j$'s take on each of the values $1$ and $-1$ with probability $1/2$, we see that
\[
\E\big[\exp(t_1Y_g+t_2Y_h)\big]=\prod_{j=1}^N\cosh(t_1g_j+t_2h_j).
\]
By comparing the Maclaurin series of $\cosh(x)$ and $\exp(x^2/2)$, we find that $\cosh(x)\le \exp(x^2/2)$. Thus
\begin{align*}
\E\big[\exp(t_1Y_g+t_2Y_h)\big]&\le\prod_{j=1}^N\exp\big(\tfrac{1}{2}(t_1g_j+t_2h_j)^2\big)\\
&=\exp\Bigg(\frac{1}{2}\sum_{j=1}^N(t_1g_j+t_2h_j)^2\Bigg),
\end{align*}
from which the desired bound easily follows.
\end{proof}
\par
We next apply Lemma~\ref{lem:subnormal} to vectors $\mathbf{g}$ and $\mathbf{h}$ whose scalar product is sufficiently small.
\begin{lemma}
\label{lem:Pr_ub}
Let $r\ge 0$ be an integer and let $\mathbf{g}$ and $\mathbf{h}$ be elements of $\{-1,1\}^{2^n}$ satisfying $\abs{\ip{\mathbf{g},\mathbf{h}}}\le 2^n/\binom{n}{r}$. Let $X$ be drawn at random from $\{-1,1\}^{2^n}$, equipped with the uniform probability measure. Write $Y_g=\ip{X,\mathbf{g}}$ and $Y_h=\ip{X,\mathbf{h}}$. Then
\[
\Pr\Big[Y_g\ge\sqrt{2^{n+1}\tbinom{n}{r}\log 2}\,\cap\, Y_h\ge\sqrt{2^{n+1}\tbinom{n}{r}\log 2}\Big]\le 4/4^{\binom{n}{r}}.
\]
\end{lemma}
\begin{proof}
Write
\[
\lambda=\sqrt{2^{n+1}\tbinom{n}{r}\log 2}
\]
and $s=\lambda/2^n$. Application of Markov's inequality gives
\begin{align*}
\Pr\big[Y_g\ge\lambda\cap Y_h\ge\lambda\big]&=\Pr\big[\exp(sY_g)\ge\exp(s\lambda)\cap \exp(sY_h)\ge\exp(s\lambda)\big]\\[1ex]
&\le\frac{\E\big[\exp(sY_g)\exp(sY_h))\big]}{[\exp(s\lambda)]^2}\\[1ex]
&\le\frac{\exp(2^ns^2(1+1/\binom{n}{r}))}{[\exp(s\lambda)]^2}
\end{align*}
by Lemma~\ref{lem:subnormal}. This last expression equals $4/4^{\binom{n}{r}}$, as required.
\end{proof}
\par
We also need the following estimate.
\begin{lemma}
\label{lem:Pr_lb}
Let $X_1,\dots,X_{2^n}$ be independent random variables taking on each of $-1$ and $1$ with probability $1/2$. Then, for all $r\ge 1$ and all sufficiently large $n$,
\[
\Pr\Big[X_1+\cdots+X_{2^n}\ge\sqrt{2^{n+1}\tbinom{n}{r}\log 2}\Big]\ge \frac{1}{3\cdot2^{\binom{n}{r}}\sqrt{\binom{n}{r}}}.
\]
\end{lemma}
\begin{proof}
A normal tail approximation of the distribution of $X_1+\cdots+X_{2^n}$ gives (see Feller~\cite[Chapter~VII, (6.7)]{Fel1968}, for example)
\[
\lim_{n\to\infty}2^{\tbinom{n}{r}}\sqrt{4\pi\tbinom{n}{r}\log 2}\,\Pr\Big[X_1+\cdots+X_{2^n}\ge \sqrt{2^{n+1}\tbinom{n}{r}\log 2}\Big]=1,
\]
from which the lemma can be deduced since $\sqrt{4\pi \log 2}<3$. 
\end{proof}


\section{Proof of Theorem~\ref{thm:main}}
\label{sec:proof}

For $g\in\RM(r,n)$, write $Y_{n,g}=2^n-2d(f_n,g)$ and
\[
Y_n=\max_{g\in\RM(r,n)}Y_{n,g},
\]
so that $Y_n=2^n-2N_r(f_n)$. Notice that
\begin{equation}
Y_{n,g}=\sum_{x\in\F_2^n}(-1)^{f_n(x)+g(x)},   \label{eqn:Y_sum_of_rv}
\end{equation}
from which we see that $Y_{n,g}$ is a sum of $2^n$ random variables, each taking each of the values $-1$ and $1$ with probability $1/2$.
\par
We make repeated use of the inequality
\begin{equation}
\Pr\big[\bigabs{Y_n-\E[Y_n]}\ge \theta\big]\le 2\exp\bigg(\!\!-\frac{\theta^2}{2^{n+1}}\bigg)\quad\text{for $\theta\ge0$},   \label{eqn:concentration}
\end{equation}
which follows from well known results on concentration of probability measures (see McDiarmid~\cite[Lemma~1.2]{McD1989}, for example). 
\par
First, we derive an upper bound for $\E[Y_n]$. Letting $s\in\R$, we have by Jensen's inequality,
\begin{align*}
\exp(s\E[Y_n])&\le\E\big[\exp(sY_n)\big]\\[1ex]
&=\E\Big[\max_{g\in\RM(r,n)}\exp(s Y_{n,g})\Big]\\[1ex]
&\le\sum_{g\in\RM(r,n)}\E\big[\exp(s Y_{n,g})\big]\\[1ex]
&\le 2^{1+\binom{n}{1}+\cdots+\binom{n}{r}}\,\exp(2^{n-1}s^2)
\end{align*}
by Lemma~\ref{lem:subnormal} with $t_1=s$ and $t_2=0$ using~\eqref{eqn:Y_sum_of_rv}. Hence
\[
\E[Y_n]\le \frac{1}{s}\big(1+\tbinom{n}{1}+\cdots+\tbinom{n}{r}\big)\,\log 2+2^{n-1}s.
\]
Now choose $s$ such that both summands are equal. This gives
\begin{equation}
\E[Y_n]\le \sqrt{2^{n+1}\big(1+\tbinom{n}{1}+\cdots+\tbinom{n}{r}\big)\log2}.   \label{eqn:E_ub}
\end{equation}
Write
\begin{equation}
\lambda_n=\sqrt{2^{n+1}\tbinom{n}{r}\log 2}   \label{eqn:def_lambda}
\end{equation}
and, for $\delta\in(0,1)$, define the set
\begin{equation}
M(\delta)=\big\{n\in\N:\E[Y_n]<(1-\delta)\lambda_n\big\}.   \label{eqn:def_N_delta}
\end{equation}
We claim that the cardinality of $M(\delta)$ is finite for all choices of $\delta>0$, which together with~\eqref{eqn:E_ub} will prove
\begin{equation}
\lim_{n\to\infty}\E[Y_n]/\lambda_n=1,   \label{eqn:proof_E}
\end{equation}
which in turn proves~\eqref{eqn:main_E}. The proof of the claim is based on an idea in~\cite{AloLitShp2010}. 
\par
Let $\alpha\in(0,1)$ be a real number, to be determined later. By Lemma~\ref{lem:frac_RM}, for all sufficiently large $n$, there exists a subset $S\subset\RM(r,n)$ satisfying
\begin{equation}
2^{(1-\alpha)\binom{n}{r}} \le\abs{S}\le 2\cdot 2^{(1-\alpha)\binom{n}{r}},   \label{eqn:size_S}
\end{equation}
say, such that
\begin{equation}
\Biggabs{\sum_{x\in\F_2^n}(-1)^{g(x)+h(x)}}\le 2^n/\tbinom{n}{r}\quad\text{for all $g,h\in S$ with $g\ne h$.}   \label{eqn:inner_prod_constraint}
\end{equation}
We have
\begin{align*}
\Pr\big[Y_n\ge\lambda_n\big]&\ge\Pr\big[\max_{g\in S}\;Y_{n,g}\ge\lambda_n\big]\\
&\ge \sum_{g\in S}\Pr\big[Y_{n,g}\ge\lambda_n\big]-\frac{1}{2}\sums{g,h\in S\\g\ne h}\Pr\big[Y_{n,g}\ge\lambda_n\,\cap\,Y_{n,h}\ge\lambda_n\big]
\end{align*}
by the Bonferroni inequality. Lemma~\ref{lem:Pr_lb} gives a lower bound for the probabilities in the first sum and, using~\eqref{eqn:Y_sum_of_rv} and~\eqref{eqn:inner_prod_constraint}, Lemma~\ref{lem:Pr_ub} gives an upper bound for the probabilities in the second sum. Applying these bounds gives, for all sufficiently large $n$,
\begin{align*}
\Pr\big[Y_n\ge \lambda_n\big]&\ge \abs{S}\cdot\frac{1}{3\cdot2^{\binom{n}{r}}\sqrt{\binom{n}{r}}}-\frac{\abs{S}^2}{2}\cdot\frac{4}{4^{\binom{n}{r}}}\\
&\ge\frac{1}{3\cdot2^{\alpha\binom{n}{r}}\sqrt{\binom{n}{r}}}-\frac{8}{4^{\alpha\binom{n}{r}}},
\end{align*}
using~\eqref{eqn:size_S}. The first term dominates the second term, so that, for all sufficiently large $n$,
\begin{equation}
\Pr\big[Y_n\ge \lambda_n\big]\ge \frac{1}{4^{\alpha\binom{n}{r}}},   \label{eqn:lb}
\end{equation}
say. By the definition~\eqref{eqn:def_N_delta} of $M(\delta)$, we have $\lambda_n>\E[Y_n]$ for all $n\in M(\delta)$. We therefore find from~\eqref{eqn:concentration} with $\theta=\lambda_n-\E[Y_n]$ that, for all $n\in M(\delta)$,
\[
\Pr\big[Y_n\ge \lambda_n\big]\le 2\exp\bigg(\!-\frac{(\lambda_n-\E[Y_n])^2}{2^{n+1}}\bigg).
\]
Comparison with~\eqref{eqn:lb} gives, for all sufficiently large $n\in M(\delta)$,
\[
\frac{1}{4^{\alpha\binom{n}{r}}}\le 2\exp\bigg(\!-\frac{(\lambda_n-\E[Y_n])^2}{2^{n+1}}\bigg),
\]
which, after rearranging and using~\eqref{eqn:def_lambda}, implies
\[
\E[Y_n]/\lambda_n\ge 1-\sqrt{1/\tbinom{n}{r}+2\alpha},
\]
By taking $\alpha=\delta^2/4$, say, we see from the definition~\eqref{eqn:def_N_delta} of $M(\delta)$ that $M(\delta)$ has finite cardinality for all $\delta\in(0,1)$, which proves~\eqref{eqn:proof_E}, and so proves~\eqref{eqn:main_E}.  
\par
To prove~\eqref{eqn:main_Pr}, we let $\e>0$ and invoke the triangle inequality to obtain
\[
\Pr\big[\abs{Y_n/\lambda_n-1}>\e\big]\le \Pr\big[\abs{Y_n-\E[Y_n]}/\lambda_n>\tfrac{1}{2}\e\big]+\Pr\big[\abs{\E[Y_n]/\lambda_n-1}>\tfrac{1}{2}\e\big].
\]
By~\eqref{eqn:proof_E}, the second probability on the right hand side equals zero for all sufficiently large $n$, and by~\eqref{eqn:concentration}, the first probability on the right hand side is at most $2\cdot 2^{-(\e^2/4)\,\binom{n}{r}}$. Hence,
\[
\sum_{n=1}^{\infty}\Pr\big[\abs{Y_n/\lambda_n-1}>\e\big]<\infty,
\]
from which and the Borel-Cantelli Lemma we conclude that 
\[
\lim_{n\to\infty} Y_n/\lambda_n=1\quad\text{almost surely}.
\]
This proves~\eqref{eqn:main_Pr}. \qed


\section*{Acknowledgement}

I thank Claude Carlet for some careful comments on a draft of this paper.


\providecommand{\bysame}{\leavevmode\hbox to3em{\hrulefill}\thinspace}
\providecommand{\MR}{\relax\ifhmode\unskip\space\fi MR }
\providecommand{\MRhref}[2]{%
  \href{http://www.ams.org/mathscinet-getitem?mr=#1}{#2}
}
\providecommand{\href}[2]{#2}

%
\end{document}